\documentclass[12pt,twoside,reqno]{amsart}
\usepackage[colorlinks=true,citecolor=blue]{hyperref}
\usepackage{mathptmx, amsmath, amssymb, amsfonts, amsthm, mathptmx, enumerate, color}
\usepackage{graphicx}
\usepackage{epstopdf}

\newtheorem{theorem}{Theorem}[section]
\newtheorem{corollary}{Corollary}[section]
\newtheorem{lemma}{Lemma}[section]

\theoremstyle{definition}

\numberwithin{equation}{section}

\begin{document}
\setcounter{page}{1}

\vspace*{1.0cm}
\title[Generalized Hilbert operators]
{Generalized Hilbert operators acting from Hardy spaces to weighted Bergman spaces}
\author[L. Wang, S. Ye]{ Liyi Wang, Shanli Ye$^{*}$}
\maketitle
\vspace*{-0.6cm}

\begin{center}
{\footnotesize {\it

School of Science, Zhejiang University of Science and Technology, Hangzhou 310023, China.

}}\end{center}

\vskip 4mm {\small\noindent {\bf Abstract.}
Let $\mu$ be a positive Borel measure on the interval $[0,1)$. For $\alpha>0$, the generalized Hankel matrix $\mathcal{H}_{\mu, \alpha}=(\mu_{n, k, \alpha})_{n, k \geq 0}$ with entries $\mu_{n, k, \alpha}=\int_{[0,1)} \frac{\Gamma(n+\alpha)}{n ! \Gamma(\alpha)} t^{n+k} \mathrm{d}\mu(t)$ induces formally the operator
\begin{equation*}
\mathcal{H}_{\mu, \alpha}(f)(z)=\sum_{n=0}^{\infty}\left(\sum_{k=0}^{\infty} \mu_{n, k, \alpha} a_k\right) z^n
\end{equation*}
on the space of all analytic function $f(z)=\sum_{k=0}^{\infty} a_{k} z^{k}$ in the unit disk $\mathbb{D}$. In this paper, we characterize the measures $\mu$ for which $\mathcal{H}_{\mu, \alpha}(f)$ is well defined on the Hardy spaces $H^p(0<p<\infty)$ and satisfies $\mathcal{H}_{\mu, \alpha}(f)(z)=\int_{[0,1)} \frac{f(t)}{(1-t z)^\alpha} \mathrm{d} \mu(t)$. Among these measures, we further describe those for which $\mathcal{H}_{\mu, \alpha}(\alpha>1)$ is a bounded (resp., compact) operator from the Hardy spaces $H^p(0<p<\infty)$ into the weighted Bergman spaces $A_{\alpha-2}^q $.

\noindent {\bf Keywords.}
Generalized Hilbert operators; Hardy space; Weighted Bergman space; Carleson measure. }

\renewcommand{\thefootnote}{}
\footnotetext{ $^*$Corresponding author.
\par
E-mail addresses: wlywwwly@163.com (L. Wang), slye@zust.edu.cn (S. Ye).
\par
((Received XX, XXXX; Accepted XX, XXXX. ))}

\section{Introduction}
Let $\mathbb{D}=\{z \in \mathbb{C}:|z|<1\}$ be the open unit disk in the complex plane $\mathbb{C}$, $H(\mathbb{D})$ denotes the class of all analytic functions on $\mathbb{D}$.

If $0<r<1$ and $f \in H(\mathbb{D})$, we set
\begin{equation*}
\aligned
& M_{p}(r, f)=\left(\frac{1}{2 \pi} \int_{0}^{2 \pi}|f(r \mathrm{e}^{i \theta})|^{p} \mathrm{d} \theta\right)^{\frac{1}{p}}, \quad 0<p<\infty . \\
& M_{\infty}(r, f)=\sup _{|z|=r}|f(z)| .
\endaligned
\end{equation*}
For $0<p \leq \infty$, the Hardy space $H^{p}$ consists of those $f \in H(\mathbb{D})$ such that
\begin{equation*}
\|f\|_{H^{p}} \stackrel{\text { def }}= \sup _{0<r<1} M_{p}(r, f)<\infty.
\end{equation*}
We refer to \cite{1} for the notation and results regarding Hardy spaces.

For $0<p<\infty$ and $\alpha>-1$, the weighted Bergman space $A_{\alpha}^{p}$ consists of those $f \in H(\mathbb{D})$ such that
\begin{equation*}
\|f\|_{A_{\alpha}^{p}} \stackrel{\text{def}}=\left((\alpha+1) \int_{\mathbb{D}}|f(z)|^{p}(1-|z|^{2})^{\alpha} \mathrm{d} A(z)\right)^{\frac{1}{p}}<\infty.
\end{equation*}
Here, $\mathrm{d} A(z)=\frac{1}{\pi} \mathrm{d} x \mathrm{d} y$ denotes the normalized Lebesgue area measure on $\mathbb{D}$. The notation $A^{p}$ is a simplified representation for the unweighted Bergman space $A_{0}^{p}$. We refer to \cite{2,3} for the notation and results about Bergman spaces.

The Bloch space $\mathcal{B}$ (cf.\cite{4}) consisting of those functions $f \in H(\mathbb{D})$ such that
\begin{equation*}
\|f\|_{\mathcal{B}} \stackrel{\text { def }}{=}|f(0)|+\sup _{z \in \mathbb{D}}(1-|z|^{2})|f^{\prime}(z)|<\infty .
\end{equation*}

For $0 \leq a < \infty$ and $0<s<\infty$, a positive Borel measure $\mu$ on $\mathbb{D}$ will be called an $a$-logarithmic $s$-Carleson measure if there exists a positive constant $C$ such that
\begin{equation*}
\sup_{I} \frac{\big(\log \frac{2 \pi}{|I|}\big)^{a} \mu(S(I))}{|I|^{s}} \leq C,
\end{equation*}
the Carleson square $S(I)$ is defined as
$S(I)=\{z=r \mathrm{e}^{i t}: \mathrm{e}^{i t} \in I ;\, 1-\frac{|I|}{2 \pi} \leq r \leq 1\}$,
where $I$ is an interval of $\partial \mathbb{D}$ and $|I|$ denotes its length.
When $a=0$, it corresponds to the classical $s$-Carleson measure, characterized by the condition $\sup_{I} \frac{\mu(S(I))}{|I|^{s}} \leq C$.
If $\mu$ satisfies $(\log \frac{2 \pi}{\mid I})^a\mu(S(I))=o(|I|^s)$ as $|I| \rightarrow 0$, we describe $\mu$ as a vanishing $a$-logarithmic $s$-Carleson measure. See \cite{5,6} for more details.

A positive Borel measure on $[0,1)$ can be seen as a Borel measure on $\mathbb{D}$ by identifying it with the measure $\tilde{\mu}$ defined by
\begin{equation*}
\tilde{\mu}(E)=\mu(E \cap [0,1))
\end{equation*}
for any Borel subset $E$ of $\mathbb{D}$. In this way, we say that a positive Borel measure $\mu$ on $[0,1)$ is an $a$-logarithmic $s$-Carleson measure on $\mathbb{D}$ if and only if there exists a constant $C>0$ such that
\begin{equation*}
\big(\log {\frac{2}{1-t}}\big)^a    \mu([t, 1)) \leq C(1-t)^{s}, \quad 0 \leq t <1.
\end{equation*}

Next, we introduce the operators discussed in this article. The study of operators induced by Hankel matrices has consistently been a popular research direction, emerging with various types of research based on different descriptions of the measure $\mu$, for instance, one can refer to \cite{7,8,9}. In this article, we consider $\mu$ as a positive Borel measure on the interval $[0,1)$.

For $\alpha>0$, we define $\mathcal{H}_{\mu, \alpha}=(\mu_{n, k, \alpha})_{n, k \geq 0}$ to be the generalized Hankel matrix with entries $\mu_{n, k, \alpha}=\int_{[0,1)}$ $\frac{\Gamma(n+\alpha)}{n ! \Gamma(\alpha)} t^{n+k} \mathrm{d} \mu(t)$. For analytic functions $f(z)=\sum_{k=0}^{\infty} a_{k} z^{k}$, the generalized Hilbert operators $\mathcal{H}_{\mu, \alpha}$ are defined by
\begin{equation}\label{Eq:1}
\mathcal{H}_{\mu, \alpha}(f)(z)=\sum_{n=0}^{\infty}\left(\sum_{k=0}^{\infty} \mu_{n, k, \alpha} a_{k}\right) z^{n},
\end{equation}
whenever the right hand side make sense and define an analytic function in $\mathbb{D}$.

The operator $\mathcal{H}_{\mu, 1}$ (namely, $\mathcal{H}_{\mu}$) has been extensively studied, for example, in \cite{10,11,12,13}. The case $\alpha=2$ has been considered in \cite{14,15,16,17,18} where the operator $\mathcal{H}_{\mu, 2}$ was marked $\mathcal{DH}_{\mu}$. In \cite{19,20}, $\mathcal{H}_{\mu, \alpha}(\alpha>0)$ were referred to as the generalized Hilbert operators.

Another generalized
integral Hilbert operator related to $\mathcal{H}_{\mu, \alpha}$, denoted by $\mathcal{I}_{\mu, \alpha}(\alpha>0)$,  is defined by
\begin{equation*}
\mathcal{I}_{\mu, \alpha}(f)(z)=\int_{[0,1)} \frac{f(t)}{(1-t z)^{\alpha}} \mathrm{d} \mu(t),
\end{equation*}
whenever the right hand side makes sense and defines an analytic function in $\mathbb{D}$.
Galanopoulos and Peláez \cite{10} characterized the measures $\mu$ for which $\mathcal{H}_{\mu, 1}$ is bounded (resp., compact) on the Hardy space $H^{1}$, and proved that $\mathcal{H}_{\mu, 1}$ is equivalent to $\mathcal{I}_{\mu, 1}$ when $\mu$ is a Carleson measure. Ye and Zhou generalized these operators in 2021, they characterized the measures $\mu$ for which $\mathcal{DH}_{\mu}$   and $\mathcal{I}_{\mu,2}$ are equivalent, and for which these operators are bounded (resp.,compact) on Bloch space \cite{14} and Bergman space \cite{15}. Later, they provided the equivalent condition for $\mathcal{H}_{\mu, \alpha}$ and $\mathcal{I}_{\mu, \alpha}$ in the case where $\alpha>0$ on Bloch space \cite{19}. Ye and Feng extended this work to weighted Bergman spaces and Dirichlet spaces in \cite{20}.

In this paper, we characterize the measures $\mu$ for which $\mathcal{H}_{\mu, \alpha}(f)\,( \alpha>0)$ is well defined on the Hardy spaces $H^p(0<p<\infty)$ and satisfies $\mathcal{H}_{\mu, \alpha}(f)(z)=\mathcal{I}_{\mu, \alpha}(f)(z)$. Then we determine the measures $\mu$ for which $\mathcal{H}_{\mu, \alpha}$ is a bounded (resp., compact) operator from the Hardy spaces $H^p(0<p<\infty)$ into the weighted Bergman spaces $A_{\alpha-2}^q (\alpha>1)$ and into the  Bloch space $\mathcal{B}$. Throughout this work, $C$ denotes a constant which may be different in each case. For any given $p>1$, $p^{\prime}$ represents the conjugate index of $p$, that is, $1 / p+1 / p^{\prime}=1$.

\section{Conditions such that $\mathcal{ H}_{\mu,\alpha}$ is well Defined on Hardy Spaces}
In this section, we establish a sufficient condition such that $\mathcal{H}_{\mu,\alpha}$ is well defined in $H^p (0<p<\infty)$. Moreover, we demonstrate that $\mathcal{H}_{\mu,\alpha}(f)=\mathcal{I}_{\mu,\alpha}(f)$ for all $f \in H^{p}$ under this condition.

\begin{lemma}\label{lm:1} \cite[p.98]{1}
If
\begin{equation*}
f(z)=\sum_{n=0}^{\infty} a_{n} z^{n} \in H^{p}, \quad 0<p \leq 1,
\end{equation*}
then $a_{n}=o(n^{1 / p-1})$, and $|a_{n}| \leq C n^{1 / p-1}\| f \|_{H^{p}}$.
\end{lemma}

\begin{lemma}\label{lm:2}  \cite[p.95]{1}
If
\begin{equation*}
f(z)=\sum_{n=0}^{\infty} a_{n} z^{n} \in H^{p}, \quad 0<p \leq 2,
\end{equation*}
then $\sum n^{p-2}|a_{n}|^{p}<\infty$, and
\begin{equation*}
\left\{\sum_{n=0}^{\infty}(n+1)^{p-2}|a_{n}|^{p}\right\}^{1/p} \leq C\|f\|_{H^{p}}.
\end{equation*}
\end{lemma}

\begin{lemma}\label{lm:3} \cite{21}
For $0<p \leq q<\infty$, $\mu$ is a $q / p$-Carleson measure if and only if there exists a positive constant $C$ such that
\begin{equation}\label{Eq:2}
\left\{\int_{[0,1)}|f(t)|^{q} \mathrm{d} \mu(t)\right\}^{\frac{1}{q}} \leq C\|f\|_{H^{p}}, \quad \text { for all } f \in H^{p}.
\end{equation}
\end{lemma}

\begin{theorem}\label{th:1}
Suppose that $0<p<\infty$, and let $\mu$ be a positive Borel measure on $[0,1)$. Then the power series in (\ref{Eq:1}) defines a well defined analytic function in $\mathbb{D}$ for every $f \in H^{p}$ in either of the two following cases:

(i) $\mu$ is a $1 / p$-Carleson measure if  $0<p \leq 1$;

(ii) $\mu$ is a 1-Carleson measure if $1<p<\infty$.

Furthermore, in such cases we have that
\begin{equation*}
\mathcal{H}_{\mu,\alpha}(f)(z)=\int_{[0,1)} \frac{f(t)}{(1-t z)^{\alpha}} \mathrm{d} \mu(t), \quad z \in \mathbb{D}.
\end{equation*}
\end{theorem}

\begin{proof}
This theorem is an extension of \cite[Theorem 3.3]{16}.

(i) Assuming $0<p \leq 1$ and $\mu$ is a $1 / p$-Carleson measure. Then Lemma \ref{lm:3} implies
\begin{equation*}
\int_{[0,1)}|f(t)| \mathrm{d} \mu(t) \leq C\|f\|_{H^{p}}, \quad \text { for all } f \in H^{p}.
\end{equation*}
Fixing $z$ with $|z|<r $ and  $f(z)=\sum_{k=0}^{\infty} a_{k} z^{k} \in H^{p}$, it follows that
\begin{equation*}
\int_{[0,1)} \frac{|f(t)|}{|1-t z|^{\alpha}} \mathrm{d} \mu(t) \leq \frac{1}{(1-r)^{\alpha}} \int_{[0,1)}|f(t)| \mathrm{d} \mu(t) \leq  \frac{C}{(1-r)^{\alpha}}\|f\|_{H^{p}}, \quad  0 <r<1 .
\end{equation*}
This implies that the integral $\int_{[0,1)} \frac{f(t)}{(1-t z)^{\alpha}}\mathrm{d}  \mu(t)$ uniformly converges on any compact subset of $\mathbb{D}$, the resulting function is analytic in $\mathbb{D}$ and, for every $z \in \mathbb{D}$,
\begin{equation}\label{Eq:3}
\mathcal{I}_{\mu, \alpha}(f)(z) =\int_{[0,1)} \frac{f(t)}{(1-t z)^{\alpha}} \mathrm{d} \mu(t)  =\sum_{n=0}^{\infty} \frac{\Gamma(n+\alpha)}{n ! \Gamma(\alpha)}\left(\int_{[0,1)} t^{n} f(t) \mathrm{d} \mu(t)\right) z^{n} .
\end{equation}
By Lemma \ref{lm:1} and \cite[Proposition 1]{11}, we have that for $f(z)=\sum_{k=0}^{\infty} a_{k} z^{k} \in H^{p}$, there exists $C>0$ such that
\begin{equation*}
\aligned
& |a_{k}| \leq C(k+1)^{1/p -1} \text { for all } n, k , \\
& |\mu_{n, k, \alpha}|=\frac{\Gamma(n+\alpha)}{n ! \Gamma(\alpha)}|\mu_{n+k}| \leq \frac{\Gamma(n+\alpha)}{n ! \Gamma(\alpha)} \frac{C}{(k+1)^{1/p}} .
\endaligned
\end{equation*}
Then we obtain that
\begin{equation*}
\aligned
\sum_{k=0}^{\infty}|\mu_{n, k, \alpha}||a_{k}| & \leq C \frac{\Gamma(n+\alpha)}{n ! \Gamma(\alpha)} \sum_{k=0}^{\infty} \frac{|a_{k}|}{(k+1)^{1 / p}}=C \frac{\Gamma(n+\alpha)}{n ! \Gamma(\alpha)} \sum_{k=0}^{\infty} \frac{|a_{k}|^{p}|a_{k}|^{1-p}}{(k+1)^{1 / p}} \\
& \leq C \frac{\Gamma(n+\alpha)}{n ! \Gamma(\alpha)} \sum_{k=0}^{\infty} \frac{|a_{k}|^{p}(k+1)^{(1/p-1)(1-p)}}{(k+1)^{1 / p}} \\
& =C \frac{\Gamma(n+\alpha)}{n ! \Gamma(\alpha)} \sum_{k=0}^{\infty}(k+1)^{p-2}|a_{k}|^{p} .
\endaligned
\end{equation*}
Combining Lemma \ref{lm:2} and (\ref{Eq:3}), it is easily yielded that the series in (\ref{Eq:1}) is well defined for all $z \in \mathbb{D}$ and that
\begin{equation*}
\aligned
\mathcal{H}_{\mu, \alpha}(f)(z) & =\sum_{n=0}^{\infty}\left(\int_{[0,1)} \frac{\Gamma(n+\alpha)}{n ! \Gamma(\alpha)} t^n f(t) \mathrm{d} \mu(t)\right) z^n=\int_{[0,1)} \sum_{n=0}^{\infty} \frac{\Gamma(n+\alpha)}{n ! \Gamma(\alpha)}(t z)^n f(t) \mathrm{d}\mu(t) \\
& =\int_{[0,1)} \frac{f(t)}{(1-t z)^\alpha} \mathrm{d} \mu(t)=\mathcal{I}_{\mu, \alpha}(f)(z), \quad z \in \mathbb{D} .
\endaligned
\end{equation*}

(ii) Suppose that $1<p<\infty$, since $\mu$ is a 1-Carleson measure and (\ref{Eq:2}) holds, the reasoning employed in the proof of (i) establishes that $\mathcal{I}_{\mu, \alpha}$ is a well defined analytic function in $\mathbb{D}$  for every $f \in H^{p}$, then we obtain that $\mathcal{H}_{\mu,\alpha}(f)=\mathcal{I}_{\mu_,\alpha}(f)$ by \cite[Theorem 3]{11}.
\end{proof}

\section{Boundedness of $\mathcal{H}_{\mu,\alpha}$ from $H^p$ into $A_{\alpha-2}^q$}
In this section, we characterize those measures $\mu$ for which $\mathcal{H}_{\mu,\alpha}$ is a bounded operator from $H^p$ into $A_{\alpha-2}^q$ for some $p$ and $q$ and into the classical Bloch space $\mathcal{B}$.

\begin{lemma}{\cite{22}}\label{lm:4}
For $0<p \leq q<\infty$ and $\alpha>-1$, then $\mu$ is a $\frac{(2+\alpha) q}{p}$-Carleson measure if and only if there exists a positive constant $C$ such that
\begin{equation*}
\left\{\int_{\mathbb{D}}|f(z)|^{q} \mathrm{d} \mu(z)\right\}^{\frac{1}{q}} \leq C\|f\|_{A_{\alpha}^{p}}, \quad \text { for all } f \in A_{\alpha}^{p}.
\end{equation*}
\end{lemma}

\begin{theorem}\label{th:2}
Suppose that $0<p<\infty$ and $\alpha>1$, let $\mu$ be a positive Borel measure on $[0,1)$ which satisfies the condition in Theorem \ref{th:1}.

(i) If $q>1$ and $q \geq \alpha p$, then $\mathcal{H}_{\mu,\alpha}$ is a bounded operator from $H^p$ into $A_{\alpha-2}^q$ if and only if $\mu$ is a $(1 / p+\alpha  / q^{\prime})$-Carleson measure;

(ii) If $q=1$ and $q \geq p$, then $\mathcal{H}_{\mu,\alpha}$ is a bounded operator from $H^p$ into $A_{\alpha-2}^1$ if and only if $\mu$ is a 1-logarithmic $1 / p$-Carleson measure;

(iii) $\mathcal{H}_{\mu,\alpha}$ is a bounded operator from $H^p$ into $\mathcal{B}$ if and only if $\mu$ is a $(1 / p + \alpha)$-Carleson measure.
\end{theorem}

\begin{proof}
Suppose that $0<p<\infty$, since $\mu$ satisfies the condition in Theorem \ref{th:1}, as in the proof of this theorem, we obtain that
\begin{equation*}
\int_{[0,1)}|f(t)| \mathrm{d} \mu(t)<\infty, \quad f \in H^{p},
\end{equation*}
then for $f \in H^{p}$ and $ g \in A_{\alpha-2}^{1}(\alpha>1)$, we have that
\begin{equation*}
\aligned
(\alpha & -1) \int_{\mathbb{D}} \int_{[0,1)}\left|\frac{f(t) g(rz)(1-|z|^{2})^{\alpha-2}}{(1-t rz)^{\alpha}}\right| \mathrm{d}\mu(t) \mathrm{d} A(z) \\
& \leq \frac{1}{(1-r)^{\alpha}} \int_{[0,1)}|f(t)| \mathrm{d} \mu(t) \int_{\mathbb{D}}(\alpha-1)(1-|z|^{2})^{\alpha-2}|g(rz)| \mathrm{d} A(z) \\
& \leq \frac{C}{(1-r)^{\alpha}}\|g_r\|_{A_{\alpha-2}^{1}}\leq \frac{C}{(1-r)^{\alpha}}\|g\|_{A_{\alpha-2}^{1}},
\endaligned
\end{equation*}
where $0 \leq r<1$ and $g_r(z)=g(rz)$. Recall the reproducing kernel function for $A_{\beta}^{p}$:
\begin{equation*}
K_{\beta}(z, w)=\frac{1}{(1-z \bar{w})^{2+\beta}}, \quad \beta >-1,
\end{equation*}
by the reproducing property, we obtain that
\begin{equation*}
f(z)=(\beta+1)\int_{\mathbb{D}}(1-|w|^{2})^{\beta} f(w) K_{\beta}(z, w) \mathrm{d} A(w), \quad f \in A_{\beta}^{1}.
\end{equation*}
Thus, using Fubini's theorem, it implies that
\begin{equation}\label{Eq:4}
\aligned
\int_{\mathbb{D}} & \overline{\mathcal{H}_{\mu, \alpha}(f)(rz)} g(rz)(1-|z|^{2})^{\alpha-2} \mathrm{d} A(z)  \\
& =\int_{\mathbb{D}} \int_{[0,1)} \frac{\overline{f(t)} \mathrm{d} \mu(t)}{(1-t r\bar{z})^{\alpha}} g(rz)(1-|z|^{2})^{\alpha-2} \mathrm{d} A(z) \\
& =\int_{[0,1)} \int_{\mathbb{D}} \frac{(1-|z|^{2})^{\alpha-2} g(rz)}{(1-t r\bar{z})^{\alpha}} \mathrm{d} A(z) \overline{f(t)} \mathrm{d} \mu(t) \\
& =\frac{1}{\alpha-1} \int_{[0,1)} \overline{f(t)} g(r^2t) \mathrm{d} \mu(t), \quad 0\leq r <1, \,f \in H^p,\, g \in A_{\alpha-2}^{1} .
\endaligned
\end{equation}

(i) Now we consider the case $q > 1$ and $q \geq \alpha p$. According to the duality theorem for $A_{\beta}^{p}$ \cite[Theorem 2.1]{23}: for $1<p< \infty$ and $\beta >-1$, we have that
$(A_{\beta}^{p})^{*} \cong A_{\gamma}^{p^{\prime}}$ and $(A_{\gamma}^{p^{\prime}})^{*} \cong A_{\beta}^{p}$ under the Cauchy pairing
\begin{equation*}
\int_{\mathbb{D}} \overline{f(z)} g(z)(1-|z|^{2})^{\frac{\beta}{p}+\frac{\gamma}{p^{\prime}}} \mathrm{d} A(z), \quad f \in A_{\beta}^{p}, \,g \in A_{\gamma}^{p^{\prime}}, \, \gamma >-1.
\end{equation*}
Using this and (\ref{Eq:4}), we obtain that $\mathcal{H}_{\mu, \alpha}$ is a bounded operator from $H^{p}$ into $A_{\alpha-2}^{q}$ if and only if there exists a positive constant $C$ such that
\begin{equation}\label{Eq:5}
\left|\int_{[0,1)} \overline{f(t)} g(t) \mathrm{d} \mu(t)\right| \leq C\|f\|_{H^p}\|g\|_{A_{\alpha-2}^{q^{\prime}}}, \quad f \in H^{p}, \,g \in A_{\alpha-2}^{q^{\prime}} .
\end{equation}
Assume that $\mathcal{H}_{\mu, \alpha}$ is a bounded operator from $H^{p}$ into $A_{\alpha-2}^{q}$. Take the test functions
\begin{equation*}
f_{a}(z)=\left(\frac{1-a^{2}}{(1-a z)^{2}}\right)^{1 / p} \quad \text { and } \quad g_{a}(z)=\left(\frac{1-a^{2}}{(1-a z)^{2}}\right)^{\alpha / q^{\prime}}, \quad 0<a<1.
\end{equation*}
A calculation implies that $f_{a} \in H^{p},\,g_{a} \in A_{\alpha-2}^{q^{\prime}}$ and
\begin{equation*}
\sup _{a \in[0,1)}\|f\|_{H^{p}}<\infty, \,\sup _{a \in[0,1)}\|g\|_{A_{\alpha-2}^{q^{\prime}}}<\infty.
\end{equation*}
It follows that
\begin{equation*}
\aligned
\infty & >C \sup _{a \in[0,1)}\|f\|_{H^{p}} \sup _{a \in[0,1)}\|g\|_{A_{\alpha-2}^{q^{\prime}}} \\
& \geq\left|\int_{[0,1)} f_{a}(t) g_{a}(t) \mathrm{d} \mu(t)\right| \\
& \geq \int_{[a, 1)}\left(\frac{1-a^{2}}{(1-a t)^{2}}\right)^{1 / p}\left(\frac{1-a^{2}}{(1-a t)^{2}}\right)^{\alpha / q^{\prime}} \mathrm{d} \mu(t) \\
& \geq \frac{1}{(1-a^{2})^{1 / p+\alpha / q^{\prime}}} \,\mu([a, 1)) .
\endaligned
\end{equation*}
This implies that $\mu$ is a $(1 / p+\alpha / q^{\prime})$-Carleson measure.

Conversely, if $\mu$ is a $(1 / p+\alpha / q^{\prime})$-Carleson measure. Let $s=1+ \alpha p / q^{\prime}$, the conjugate exponent of s is $s^{\prime}=1+ q^{\prime} / \alpha p$ and $1 / p+\alpha / q^{\prime} =s / p =\alpha s^{\prime} / q^{\prime}$. By Lemma \ref{lm:3} and Lemma \ref{lm:4}, we obtain that there exists a constant $C>0$ such that
\begin{equation*}
\left(\int_{[0,1)}|f(t)|^{s} \mathrm{d} \mu(t)\right)^{1 / s} \leq C\|f\|_{H^{p}}, \quad \text { for all } f \in H^{p}
\end{equation*}
and
\begin{equation*}
\left(\int_{[0,1)}|g(t)|^{s^{\prime}} \mathrm{d} \mu(t)\right)^{1 / s^{\prime}} \leq C\|g\|_{A_{\alpha-2}^{q^{\prime}}}, \quad \text { for all } g \in A_{\alpha-2}^{q^{\prime}}.
\end{equation*}
Subsequently, applying Hölder's inequality, we can observe that
\begin{equation*}
\aligned
\int_{[0,1)}|f(t)||g(t)| \mathrm{d} \mu(t) & \leq\left(\int_{[0,1)}|f(t)|^{s} \mathrm{d} \mu(t)\right)^{1 / s}\left(\int_{[0,1)}|g(t)|^{s^{\prime}} \mathrm{d} \mu(t)\right)^{1 / s^{\prime}} \\
& \leq C\|f\|_{H^{p}}\|g\|_{A_{\alpha-2}^{q^{\prime}}} .
\endaligned
\end{equation*}
Hence, (\ref{Eq:5}) holds and then $\mathcal{H}_{\mu, \alpha}$ is a bounded operator from $H^{p}$ into $A_{\alpha-2}^{q}$.

(ii) Let us recall the duality theorem for $A_{\alpha-2}^{1}$ (see \cite[Theorem 5.15]{3}): $(\mathcal{B}_{0})^{*} \cong A_{\alpha-2}^{1}$ and $(A_{\alpha-2}^{1})^{*} \cong \mathcal{B}$ under the pairing
\begin{equation*}
\lim_{r \rightarrow 1^-}\int_{\mathbb{D}} \overline{f(rz)} g(rz)(1-|z|^{2})^{\alpha-2} \mathrm{d} A(z), \quad f \in \mathcal{B}\,(\text {resp., } \mathcal{B}_{0}), \,g \in A_{\alpha-2}^{1}.
\end{equation*}
This and (\ref{Eq:4}) imply that $\mathcal{H}_{\mu, \alpha}$ is a bounded operator from $H^{p}$ into $A_{\alpha-2}^{1}$ if and only if there exists a positive constant $C$ such that
\begin{equation*}
\left|\int_{[0,1)} \overline{f(t)} g(r^2t) \mathrm{d} \mu(t)\right| \leq C\|f\|_{H^{p}}\|g\|_{\mathcal{B}}, \quad f \in H^{p},\, g \in \mathcal{B}_{0} .
\end{equation*}
Assume that $\mathcal{H}_{\mu, \alpha}$ is a bounded operator from $H^{p}$ into $A_{\alpha-2}^{1}$. Take the test functions
\begin{equation*}
f_{a}(z)=\left(\frac{1-a^{2}}{(1-a z)^{2}}\right)^{1 / p} \quad \text { and } \quad g_{a}(z)=\log \frac{2}{1-a z}, \quad 0<a<1 .
\end{equation*}
A calculation gives that $g_{a} \in \mathcal{B}_{0}$ and $\sup _{a \in[0,1)}\|g\|_{\mathcal{B}}<\infty$.
Taking $r\in [a,1)$, it follows that
\begin{equation*}
\aligned
\infty & >C \sup _{a \in[0,1)}\|f\|_{H^{p}} \sup _{a \in[0,1)}\|g\|_{\mathcal{B}} \\
& \geq\left|\int_{[0,1)} f_{a}(t) g_{a}(r^2t) \mathrm{d} \mu(t)\right| \\
& \geq \int_{[a, 1)}\left(\frac{1-a^{2}}{(1-a t)^{2}}\right)^{1 / p} \log \frac{2}{1-ar^2t} \mathrm{d} \mu(t) \\
& \geq \frac{\log \frac{2}{1-a}}{(1-a)^{1 / p}} \,\mu([a, 1)) .
\endaligned
\end{equation*}
This implies that $\mu$ is a 1-logarithmic $1/p$-Carleson measure.

Conversely, assume that $\mu$ is a 1-logarithmic $1/p$-Carleson measure. As is well known, any function $g \in \mathcal{B}$\cite{4} possesses the following property:
\begin{equation*}
|g(z)| \leq C\|g\|_{\mathcal{B}} \log \frac{2}{1-|z|}.
\end{equation*}
Thus, if $0<p \leq 1$, for every $f \in H^{p}$ and $g \in \mathcal{B}_{0} \subset \mathcal{B}$, we have
\begin{equation*}
\left|\int_{[0,1)} \overline{f(t)} g(r^2t) \mathrm{d} \mu(t)\right| \leq C\|g\|_{\mathcal{B}} \int_{[0,1)}|f(t)| \log \frac{2}{1-t} \mathrm{d} \mu(t).
\end{equation*}
Let $\mathrm{d} \nu(t)=\log \frac{2}{1-t} \mathrm{d} \mu(t)$, by \cite[Proposition 2.5]{12}, we get that $\nu$ is a $1/p$-Carleson measure. Hence, Lemma \ref{lm:3} implies that
\begin{equation*}
\aligned
\left|\int_{[0,1)} \overline{f(t)} g(r^2t) \mathrm{d} \mu(t)\right|  &
\leq C\|g\|_{\mathcal{B}} \int_{[0,1)}|f(t)| \mathrm{d} \nu(t) \\
& \leq  C\|f\|_{H^{p}}\|g\|_{\mathcal{B}}, \quad f \in H^{p},\, g \in \mathcal{B}_{0} .
\endaligned
\end{equation*}

(iii) Using (\ref{lm:4}) and the duality theorem for $\mathcal{B}$, we obtain that $\mathcal{H}_{\mu,\alpha}$ is a bounded operator from $H^p$ into $\mathcal{B}$ if and only if
\begin{equation*}
\left|\int_{[0,1)} \overline{f(t)} g(r^2t) \mathrm{d} \mu(t)\right| \leq C\|f\|_{H^p}\|g\|_{A_{\alpha-2}^1}, \quad f \in H^p,\, g \in A_{\alpha-2}^1 .
\end{equation*}
Assume that $\mathcal{H}_{\mu, \alpha}$ is a bounded operator from $H^{p}$ into $\mathcal{B}$. Take the test functions
\begin{equation*}
f_{a}(z)=\left(\frac{1-a^{2}}{(1-a z)^{2}}\right)^{1 / p} \quad \text { and } \quad g_{a}(z)=\left(\frac{1-a^{2}}{(1-a z)^{2}}\right)^{\alpha}, \quad 0<a<1 .
\end{equation*}
A calculation gives that $g_{a} \in  A_{\alpha-2}^1$ and $\sup _{a \in[0,1)}\|g\|_{ A_{\alpha-2}^1}<\infty$.
For any $r \in [a,1)$, it follows that
\begin{equation*}
\aligned
\infty & >C \sup _{a \in[0,1)}\|f\|_{H^{p}} \sup _{a \in[0,1)}\|g\|_{A_{\alpha-2}^1} \\
& \geq\left|\int_{[0,1)} f_{a}(t) g_{a}(r^2t) \mathrm{d} \mu(t)\right| \\
& \geq \int_{[a, 1)}\left(\frac{1-a^{2}}{(1-a t)^{2}}\right)^{1 / p}\left(\frac{1-a^{2}}{(1-a r^2t)^{2}}\right)^{\alpha}  \mathrm{d} \mu(t) \\
& \geq \frac{1}{(1-a)^{1 / p+\alpha}}
 \,\mu([a, 1)) .
\endaligned
\end{equation*}
This implies that $\mu$ is a  $(1/p+\alpha)$-Carleson measure.

Conversely, assume that $\mu$ is a  $(1/p+\alpha)$-Carleson measure. It is well known that any function $f \in H^p$\cite{1} has the property
\begin{equation*}
|f(z)| \leq C \frac{\|f\|_{H^{p}}}{(1-|z|)^{1 / p}}.
\end{equation*}
Thus, we have
\begin{equation*}
\left|\int_{[0,1)} \overline{f(t)} g(r^2t) \mathrm{d} \mu(t)\right| \leq C\|f\|_{H^p} \int_{[0,1)}\frac{1}{{(1-t})^{1/p}} |g(r^2t)|\mathrm{d} \mu(t).
\end{equation*}
Similar to (ii), let $\mathrm{d} \nu(t)=\frac{1}{{(1-t})^{1/p}} \mathrm{d} \mu(t)$, then $\nu$ is an $\alpha$-Carleson measure. By Lemma \ref{lm:4}, we obtain that
\begin{equation*}
\aligned
\left|\int_{[0,1)} \overline{f(t)} g(r^2t) \mathrm{d} \mu(t)\right| \leq &
C\|f\|_{H^p} \int_{[0,1)} |g(r^2t)|\mathrm{d} \nu(t) \\
\leq & C\|f\|_{H^{p}}\|g\|_{A_{\alpha-2}^1}, \quad f \in H^{p},\, g \in A_{\alpha-2}^1 .
\endaligned
\end{equation*}
\end{proof}

\begin{corollary}
Suppose that $0<p<\infty$ and let $\mu$ be a positive Borel measure on $[0,1)$ which satisfies the condition in Theorem \ref{th:1}.

(i) If $q>1$ and $q \geq 2p$, then $\mathcal{DH}_{\mu}$ is a bounded operator from $H^p$ into $A^q$ if and only if $\mu$ is a $(1 / p+2 / q^{\prime})$-Carleson measure;

(ii) If $q=1$ and $q \geq p$, then $\mathcal{DH}_{\mu}$ is a bounded operator from $H^p$ into $A^1$ if and only if $\mu$ is a 1-logarithmic $1 / p$-Carleson measure;

(iii) $\mathcal{DH}_{\mu}$ is a bounded operator from $H^p$ into $\mathcal{B}$ if and only if $\mu$ is a $(1 / p+2)$-Carleson measure.
\end{corollary}

\section{Compactness of $\mathcal{H}_{\mu,\alpha}$ from $H^p$ into $A_{\alpha-2}^q$}
Next, we characterize the measures $\mu$ such that $\mathcal{H}_{\mu,\alpha}$ is compact from $H^p$ into $A_{\alpha-2}^q$ and into $\mathcal{B}$. We begin by presenting a commonly used lemma for dealing with compactness.
\begin{lemma}\label{lm:5}
For $0<p,q<\infty$ and $\alpha>1$, suppose that $\mathcal{H}_{\mu,\alpha}$ is a bounded operator from $H^p$  into $A_{\alpha-2}^q$. Then $\mathcal{H}_{\mu,\alpha}$ is a compact operator if and only if $\mathcal{H}_{\mu,\alpha}(f_n) \rightarrow 0 $ in $A_{\alpha-2}^q$, for any bounded sequence $\{f_n\}$ in $H^p$  which converges to 0 uniformly on every compact subset of $\mathbb{D}$.
\end{lemma}
\begin{proof}
The proof is similar to that of \cite[Proposition 3.11]{24} and details are omitted.
\end{proof}

\begin{theorem}\label{th:3}
Suppose that $0<p<\infty$ and $\alpha>1$, let $\mu$ be a positive Borel measure on $[0,1)$ which satisfies the condition in Theorem \ref{th:1}.

(i) If $q>1$ and $q \geq \alpha p$, then $\mathcal{H}_{\mu,\alpha}$ is a compact operator from $H^p$ into $A_{\alpha-2}^q$ if and only if $\mu$ is a vanishing $(1 / p+\alpha  / q^{\prime})$-Carleson measure;

(ii) If $q=1$ and $q \geq p$, then $\mathcal{H}_{\mu,\alpha}$ is a compact operator from $H^p$ into $A_{\alpha-2}^1$ if and only if $\mu$ is a vanishing 1-logarithmic $1 / p$-Carleson measure;

(iii) $\mathcal{H}_{\mu,\alpha}$ is a compact operator from $H^p$ into $\mathcal{B}$ if and only if $\mu$ is a vanishing $(1 / p + \alpha)$-Carleson measure.
\end{theorem}

\begin{proof}
Firstly, let $0<p \leq q<\infty$, if $\mu$ is a $q / p$-Carleson measure, then the identity mapping $i$ is well defined from $H^{p}$ into $L^{q}(\mathbb{D}, \mu)$ by Lemma \ref{lm:3}, we denote the norm of $i$ by $\mathcal{N}_{1}(\mu)$. For $0<r<1$, let
\begin{equation*}
\mathrm{d} \mu_{r}(z)=\chi_{r<|z|<1}(t) \mathrm{d} \mu(t),
\end{equation*}
then $\mu$ is a vanishing $q / p$-Carleson measure if and only if
\begin{equation}\label{Eq:6}
\mathcal{N}_{1}(\mu_{r}) \rightarrow 0 \quad \text { as } r \rightarrow 1^{-}.
\end{equation}

Similarly, according to Lemma \ref{lm:4}, the identity mapping $j$ is well defined from $A_{\alpha-2}^{p}$ into $L^{q}(\mathbb{D}, \mu)$ if $\mu$ is an $\alpha q / p$-Carleson measure ($\alpha>1$). Let $\mathcal{N}_{2}(\mu)$ be the norm of $j$, then for the above defined $\mathrm{d} \mu_{r}$, we obtain that $\mu$ is a vanishing $\alpha  q / p$-Carleson measure if and only if
\begin{equation}\label{Eq:7}
\mathcal{N}_{2}(\mu_{r}) \rightarrow 0 \quad \text { as } r \rightarrow 1^{-}.
\end{equation}

(i) Assume that $q>1$, $q \geq \alpha p$ and $\mathcal{H}_{\mu, \alpha}$ is a compact operator from $H^{p}$ into $A_{\alpha-2}^{q}$. Let $\{a_{n}\} \subset(0,1)$ be any sequence with $a_{n} \rightarrow 1$. Take
\begin{equation*}
f_{a_{n}}(z)=\left(\frac{1-a_{n}^{2}}{(1-a_{n} z)^{2}}\right)^{1 / p}, \quad z \in \mathbb{D},
\end{equation*}
then  $f_{a_{n}}(z) \in H^{p},\, \sup _{n \geq 1}\|f_{a_{n}}\|_{H^{p}}<\infty$ and $f_{a_{n}}\rightarrow0$ uniformly on any compact subset of $\mathbb{D}$. Using Lemma \ref{lm:5}, we have that $\{\mathcal{H}_{\mu, \alpha}(f_{a_{n}})\}$ converges to 0 in $A_{\alpha-2}^{q}$. This and (\ref{Eq:4}) imply that
\begin{equation}\label{Eq:8}
\lim _{n \rightarrow \infty} \int_{[0,1)} \overline{f_{a_{n}}(t)} g(t) \mathrm{d} \mu(t)
=(\alpha-1)\lim _{n \rightarrow \infty} \int_{\mathbb{D}} \overline{\mathcal{H}_{\mu, \alpha}(f_{a_{n}})(z)} g(z) (1-|z|^{2})^{\alpha-2} \mathrm{d}A(z)=0,
\end{equation}
where $g \in A_{\alpha-2}^{q^{\prime}}$. Now we take
\begin{equation*}
g_{a_{n}}(z)=\left(\frac{1-a_{n}^{2}}{(1-a_{n} z)^{2}}\right)^{\alpha / q^{\prime}} \in A_{\alpha-2}^{q^{\prime}}, \quad z \in \mathbb{D},
\end{equation*}
then we have that
\begin{equation*}
\aligned
& \int_{[0,1)} \overline{f_{a_{n}}(t)} g_{a_{n}}(t) \mathrm{d} \mu(t) \\
& \quad \geq \int_{a_{n}}^{1}\left(\frac{1-a_{n}^{2}}{(1-a_{n} t)^{2}}\right)^{1 / p}\left(\frac{1-a_{n}^{2}}{(1-a_{n} t)^{2}}\right)^{\alpha / q^{\prime}} \mathrm{d} \mu(t) \\
& \quad \geq \frac{C}{(1-{a_{n}}^2)^{1/ p+\alpha/ q^{\prime}}}\, \mu([a_{n}, 1)) .
\endaligned
\end{equation*}
Using (\ref{Eq:8}) and $\{a_{n}\} \subset(0,1)$ is a sequence with $a_{n} \rightarrow 1$, we obtain that
\begin{equation*}
\lim _{a \rightarrow 1^{-}} \frac{1}{(1-a)^{1 / p+\alpha / q^{\prime}}}\, \mu([a, 1))=0,
\end{equation*}
which implies that $\mu$ is a vanishing $(1 / p+\alpha / q^{\prime})$-Carleson measure.

Conversely, suppose that $\mu$ is a vanishing $(1 / p+\alpha  / q^{\prime})$-Carleson measure. Let $\{f_{n}\}_{n=1}^{\infty}$ be a sequence of $H^{p}$ functions with $\sup _{n \geq 1}\|f_{n}\|_{H^{p}}<\infty$, and $\{f_{n}\} $ converges to 0 uniformly on every compact subset of $\mathbb{D}$, it suffices to prove that $\mathcal{H}_{\mu, \alpha}(f_{n}) \rightarrow 0$ in $A_{\alpha-2}^{q}$ by Lemma \ref{lm:5}. Taking $g \in A_{\alpha-2}^{q^{\prime}}$, we have that
\begin{equation*}
\int_{[0,1)}|f_{n}(t)||g(t)| \mathrm{d} \mu(t)=\int_{[0, r)}|f_{n}(t)||g(t)| \mathrm{d} \mu(t)+\int_{[r, 1)}|f_{n}(t)||g(t)| \mathrm{d} \mu(t) .
\end{equation*}
Since $\{f_{n}\} \rightarrow 0$ uniformly on compact subsets of $\mathbb{D}$, we obtain that
\begin{equation*}
\int_{[0, r)}|f_{n}(t)||g(t)| \mathrm{d} \mu(t) \rightarrow 0.
\end{equation*}
Based on the proof of Theorem \ref{th:2}, part(i), we have that
\begin{equation*}
\aligned
\int_{[r, 1)}|f_{n}(t)||g(t)| \mathrm{d} \mu(t) & \leq\left(\int_{[0,1)}|f_{n}(t)|^{s} \mathrm{d} \mu_{r}(t)\right)^{1 / s}\left(\int_{[0,1)}|g(t)|^{s^{\prime}} \mathrm{d} \mu_{r}(t)\right)^{1 / s^{\prime}} \\
& \leq\mathcal{N}_{1}(\mu_{r})\mathcal{N}_{2}(\mu_{r})\|f\|_{H^{p}}\|g\|_{A_{\alpha-2}^{q^{\prime}}}.
\endaligned
\end{equation*}
By (\ref{Eq:6}) and (\ref{Eq:7}),
\begin{equation*}
\lim _{n \rightarrow \infty} \int_{[0,1)}|f_{n}(t) || g(t)| \mathrm{d} \mu(t)=0, \quad \text { for all } g \in A_{\alpha-2}^{q^{\prime}}.
\end{equation*}
Thus
\begin{equation*}
\lim _{n \rightarrow \infty}\left|\int_{\mathbb{D}} \overline{\mathcal{H}_{\mu, \alpha}(f_{n})(z)} g(z)(1-|z|^{\alpha-2}) \mathrm{d} A(z)\right|=0, \quad \text { for all } g \in A_{\alpha-2}^{q^{\prime}}.
\end{equation*}
This implies that $\mathcal{H}_{\mu, \alpha}(f_{n}) \rightarrow 0$ in $A_{\alpha-2}^{q}$, then we obtain $\mathcal{H}_{\mu, \alpha}$ is a compact operator from $H^{p}$ into $A_{\alpha-2}^{q}$.

(ii) Let $0<p \leq q < \infty$ and $q=1$, assuming $\mathcal{H}_{\mu, \alpha}$ is a compact operator from $H^{p}$ into $A_{\alpha-2}^{1}$. For any sequence  $\{a_{n}\} \subset(0,1)$ with $a_{n} \rightarrow 1$, using the function $f_{a_{n}}$ as in the proof of (i), then $\{\mathcal{H}_{\mu, \alpha}(f_{a_{n}})\} \rightarrow 0$ in $A_{\alpha-2}^{1}$. In this way, we obtain \begin{equation}\label{Eq:9}
\lim _{n \rightarrow \infty} \int_{[0,1)} \overline{f_{a_{n}}(t)} g(r^2t) \mathrm{d} \mu(t)
=(\alpha-1)\lim _{n \rightarrow \infty} \int_{\mathbb{D}} \overline{\mathcal{H}_{\mu, \alpha}(f_{a_{n}})(rz)} g(rz) (1-|z|^{2})^{\alpha-2} \mathrm{d}A(z)=0,
\end{equation}
where $g \in \mathcal{B}_{0}$. Take $g_{a_{n}}(z)=\log \frac{2}{1-a_{n} z}\in \mathcal{B}_{0}$,  for $r \in (a_n ,1)$, we infer that
\begin{equation*}
\aligned
& \int_{[0,1)} \overline{f_{a_{n}}(t)} g_{a_{n}}(r^2t) \mathrm{d} \mu(t) \\
& \quad \geq C \int_{[a_{n}, 1)}\left(\frac{1-a_{n}^{2}}{(1-a_{n} t)^{2}}\right)^{1 / p} \log \frac{2}{1-a_{n} r^2t} \mathrm{d} \mu(t) \\
& \quad \geq C \frac{ \log \frac{2}{1-a_{n}}}{(1-{a_{n}})^{1/ p}}\, \mu([a_{n}, 1)) .
\endaligned
\end{equation*}
Letting $n \rightarrow \infty$ and we have
\begin{equation*}
\lim _{a \rightarrow 1^{-}} \frac{\log \frac{2}{1-a}}{(1-a)^{1 / p}} \,\mu([a, 1))=0.
\end{equation*}
Thus $\mu$ is a vanishing 1-logarithmic $1 / p$-Carleson measure.

Conversely, suppose that $\mu$ is a vanishing 1-logarithmic $1 / p$-Carleson measure, we obtain that $\mathrm{d}  \nu (t) = \log \frac{2}{1-t}\mathrm{d} \mu(t)$ is a vanishing $1 / p$-Carleson measure. Using the same function sequence $\{f_{n}\}_{n=1}^{\infty}$ in the proof of (i), then we only need to determine that $\{\mathcal{H}_{\mu, \alpha}(f_{n})\} \rightarrow 0$ in $A_{\alpha-2}^{1}$.
For every $g \in \mathcal{B}_{0}$ and $0<r<1$, we get that
\begin{equation*}
\aligned
\int_{[0,1)}|f_{n}(t)||g(t)| \mathrm{d} \mu(t)= & \int_{[0, r)}|f_{n}(t)||g(t)| \mathrm{d} \mu(t)+\int_{[r, 1)}|f_{n}(t)||g(t)| \mathrm{d} \mu(t) \\
\leq & \int_{[0, r)}|f_{n}(t)||g(t)| \mathrm{d} \mu(t) +C\|g\|_{\mathcal{B}} \int_{[r, 1)}|f_{n}(t)| \log \frac{2}{1-t} \mathrm{d} \mu(t) \\
\leq & \int_{[0, r)}|f_{n}(t)||g(t)| \mathrm{d} \mu(t)+C\|g\|_{\mathcal{B}} \int_{[0,1)}|f_{n}(t)| \mathrm{d} v_{r}(t) \\
\leq & \int_{[0, r)}|f_{n}(t)||g(t)| \mathrm{d} \mu(t)+C \mathcal{N}_{1}(v_{r}) \|f_{n}\|_{H^{p}}\|g\|_{\mathcal{B}}.
\endaligned
\end{equation*}
Then we get that
\begin{equation*}
\lim _{n \rightarrow \infty} \int_{[0,1)}|f_{n}(t)  || g(t)| \mathrm{d} \mu(t)=0, \quad \text { for all } g \in \mathcal{B}_{0}.
\end{equation*}
Using this and (\ref{Eq:9}), we have
\begin{equation*}
\lim_{n \rightarrow \infty} \left(\lim _{r \rightarrow 1^-}\left|\int_{\mathbb{D}} \overline{\mathcal{H}_{\mu, \alpha}(f_{n})(rz)} g(rz)(1-|z|^{2})^{\alpha-2} \mathrm{d} A(z)\right|\right)=0, \quad \text { for all } g \in \mathcal{B}_{0}.
\end{equation*}
This implies $\mathcal{H}_{\mu, \alpha}(f_{n}) \rightarrow 0$ in $A_{\alpha-2}^{1}$ and then we establish that $\mathcal{H}_{\mu, \alpha}$ is a compact operator from $H^{p}$ into $A_{\alpha-2}^{1}$.

The method used in (iii) is similar to (ii), we omit the details here.
\end{proof}

\begin{corollary}
Suppose that $0<p<\infty$ and let $\mu$ be a positive Borel measure on $[0,1)$ which satisfies the condition in Theorem \ref{th:1}.

(i) If $q>1$ and $q \geq 2p$, then $\mathcal{DH}_{\mu}$ is a compact operator from $H^p$ into $A^q$ if and only if $\mu$ is a vanishing $(1 / p+2 / q^{\prime})$-Carleson measure;

(ii) If $q=1$ and $q \geq p$, then $\mathcal{DH}_{\mu}$ is a compact operator from $H^p$ into $A^1$ if and only if $\mu$ is a vanishing 1-logarithmic $1 / p$-Carleson measure;

(iii) $\mathcal{DH}_{\mu}$ is a compact operator from $H^p$ into $\mathcal{B}$ if and only if $\mu$ is a vanishing $(1 / p+2)$-Carleson measure.
\end{corollary}

\vskip 6mm
\noindent{\bf Acknowledgments}

\noindent   The research was supported by Zhejiang Provincial Natural Science Foundation (Grant No.
 LY23A010003) and National Natural Science Foundation of China (Grant No. 11671357).

\end{document}